%% file: paper.tex
\documentclass[11pt]{article}

\usepackage[margin=4cm]{geometry}
\usepackage[utf8]{inputenc}
\usepackage[T1]{fontenc}
\usepackage{authblk}
\usepackage{graphicx}
\usepackage{grffile}
\usepackage{longtable}
\usepackage{wrapfig}
\usepackage{rotating}
\usepackage[normalem]{ulem}
\usepackage{amsmath}
\usepackage{textcomp}
\usepackage{amssymb}
\usepackage{capt-of}
\usepackage{hyperref}
\usepackage{url}

\usepackage{mathtools}

\usepackage[nolist, nohyperlinks]{acronym}
\usepackage{amsmath}
\usepackage{amssymb}
\usepackage{amsthm}
\usepackage{stmaryrd}
\usepackage{todonotes}
\usepackage{ebproof}
\usepackage{pdflscape}
\usepackage{enumitem}

\usepackage{fixme}
\fxsetup{layout=footnote}
\usepackage{cmll}
\usepackage{enumitem}
\usepackage[capitalize]{cleveref}

\DeclareUrlCommand\emailUrl{\urlstyle{rm}}
\newcommand{\email}[1]{\href{mailto:#1}{\protect\emailUrl{#1}}}
\newcommand{\emailVierling}{\email{jannik.vierling@tuwien.ac.at}}
\newcommand{\emailHetzl}{\email{stefan.hetzl@tuwien.ac.at}}

\newtheorem*{theorem*}{Theorem}
\newtheorem{theorem}{Theorem}[section]
\newtheorem{proposition}[theorem]{Proposition}
\newtheorem{definition}[theorem]{Definition}
\newtheorem{lemma}[theorem]{Lemma}
\newtheorem{corollary}[theorem]{Corollary}
\newtheorem{conjecture}[theorem]{Conjecture}
\newtheorem*{conjecture*}{Conjecture}
\newtheorem{remark}[theorem]{Remark}

\input{header.tex}

\newcommand{\DocumentTitle}{
  Quantifier-free induction for lists
}

\author[1,2]{Stefan Hetzl\footnote{corresponding author}}
\author[1,3]{Jannik Vierling}
\affil[1]{Vienna University of Technology\protect\\Institute of Discrete Mathematics and Geometry}
\affil[2]{\emailHetzl}
\affil[3]{\emailVierling}
\date{}
\title{\DocumentTitle}

\hypersetup{
 pdfauthor={Jannik Vierling},
 pdftitle={\DocumentTitle},
 pdfkeywords={},
 pdfsubject={},
 pdfcreator={}, 
 pdflang={English}}

\begin{document}

\maketitle

\begin{abstract}
  We investigate quantifier-free induction for Lisp-like lists constructed inductively from the empty list \(\Nil\) and the operation \(\ConsSymbol\), that adds an element to the front of a list.
  First we show that, for \(m \geq 1\), quantifier-free \(m\)-step induction does not simulate quantifier-free \((m + 1)\)-step induction.
  Secondly, we show that for all \(m \geq 1\), quantifier-free \(m\)-step induction does not prove the right cancellation property of the concatenation operation on lists defined by left-recursion.
\end{abstract}


\section{Introduction}

% Research program

In this article we consider Lisp-like lists in the context of the automation of proof by mathematical induction.
% Lisp-like lists are finite sequences that are inductively constructed from the operators \(\Nil\), the empty sequence, and \(\ConsSymbol\), the operation that adds an element, from a possibly infinite set of elements, to the front of a sequence.
The subject of \ac{AITP} aims at automating the process of proving statements about inductively constructed objects such as natural numbers, lists and trees.
The formal verification of software is a particularly prominent application of automated inductive theorem proving.
Since every non-trivial program contains loops or recursion, some form of mathematical induction is necessary to reason about such programs.
By Gödel's incompleteness theorem the task addressed by \ac{AITP} is in general not even semi-decidable.
Therefore, there is a lot more freedom in the choice of the proof systems
 than in the case of first-order validity.
For that reason and because of technical constraints, a great variety of methods have been developed for that purpose.
To name just a few examples, there are methods based on recursion analysis \cite{bundy1989}, integration into saturation-based provers \cite{reger2019,kersani2013,cruanes2017}, cyclic proofs \cite{brotherston2012}, theory exploration \cite{claessen2013a}, proof by consistency \cite{comon2001}.

The current methodology in automated inductive theorem proving concentrates primarily on the implementation of systems and their empirical evaluation.
The work in this article is part of a research program that aims at complementing this state of the art by focusing on the formal analysis of methods for automated inductive theorem proving.
In particular, we aim at understanding the theoretical limits of systems by developing upper bounds on the logical strength of methods.
Establishing sufficiently tight upper bounds on the strength of \ac{AITP} systems often allows us to provide practically meaningful unprovability results whereas an empirical evaluation only shows the failure of a particular implementation.
Moreover, upper bounds typically reveal the particular form of induction underlying the \ac{AITP} systems.
This knowledge permits the direct comparison of methods and helps in judging the applicability of \ac{AITP} systems to certain domains.

So far the work in this research program \cite{wong2018,hetzl2020,hetzl2022,hetzl2023apal,vierling2022} has concentrated on induction for natural numbers only.
However, since lists and other inductive data types are fundamental structures of computer science, it is of paramount importance for the subject of \ac{AITP} to analyze the mechanical properties of these inductive datatypes.
In this article we make a first important step towards extending this research program to inductively defined lists.
In particular, we show that the right cancellation property of the concatenation of
 lists is not provable by a form of induction used in some automated inductive theorem proving systems.
With this result we pave the way for obtaining further unprovability results for \ac{AITP} systems on lists and other inductive data types.

% Lists and why they are interesting.

% Lisp-like lists are an inductive construction of finite sequences from the operators \(\Nil\), the empty list, and \(\ConsSymbol\), the operation that adds an element, from a possibly infinite set of elements, to the front of a list.
% These lists are an important datatype in computer science.
% Lists and similar inductive constructions are fundamental datatypes in functional programming languages.
% Furthermore, programming languages in the Lisp family \cite{mccarthy1960recursive} base their syntax and their semantics on this type of lists thereby achieving homoiconicity.
% More generally lists of one form or another play an important role in the semantics of programming languages.
% Therefore, reasoning about axiomatic formulations of lists and similar structures is a highly important concern for the field of automated verification of software.

% Related work on lists

In the following we briefly mention some aspects of axiomatic theories of finite lists have been studied in theoretical computer science.
In \cite{moore1981} an axiomatic theory of linear lists (Lisp-like lists) is defined and some basic results about consistency, completeness, and independency of the axioms are shown.
Similar theories are considered in a more general setting in \cite{oppen1978}.
In \cite{goncharov1986,bazhenov2015,aleksandrova2019list} the computability aspects of list structures are investigated.

Axiomatic theories of lists are closely related to theories of concatenation studied in logic \cite{tarski1935,quine1946}.
Theories of concatenation axiomatise strings of symbols over a finite alphabet.
Theories of concatenation have been proposed as alternative basic systems for the
 development of metamathematical results such as Gödel's incompleteness theorems and
 computability \cite{visser2006,grzegorczyk2005,grzegorczyk2008,corcoran1974,thatcher1966}.
In such theories there is no need to develop a coding of finite
 sequences \cite{quine1946,grzegorczyk2005}.
Hence, theories of concatenation permit a more natural development of syntax.

% This article and structure.

In this article we consider the provability of the right-cancellation of the concatenation of finite lists from quantifier-free big-step first-order induction for Lisp-like lists.
After recalling some basic concepts and notations in \cref{sec:preliminaries},
  we show the two main results of this article in \cref{sec:two_step,sec:right_cancellation}.
First, in \cref{sec:two_step}, we show that in general \(m\)-step quantifier-free induction does not prove \((m+1)\)-step quantifier-free induction.
This results sets induction on lists in contrast with induction for natural numbers where big-step quantifier-free induction is not stronger than one-step quantifier-free induction.
Secondly, in \cref{sec:right_cancellation}, we show that for all \(m \geq 1\), \(m\)-step induction, over the language consisting of the list constructors and a concatenation operator, does not prove the right cancellation property of the concatenation operation.
In order to show these unprovability results we will construct models whose domain contains
 sequences of transfinite length.

\section{Preliminaries}
\label{sec:preliminaries}

In this section we introduce some concepts, notations, and results that we will use throughout the article.
In \cref{sec:preliminaries:logic} we recall some basic concepts and notations of many-sorted first-order logic.
\Cref{sec:preliminaries:induction} defines some basic axioms of the list constructors and the traditional induction schema for lists as well as related terminology.
Finally, in \cref{sec:preliminaries:sequences} we introduce some concepts on transfinite sequences, which we will use in the model theoretic constructions of \cref{sec:two_step} and \cref{sec:right_cancellation}.

\subsection{Many-sorted first-order logic}
\label{sec:preliminaries:logic}
We work in the setting of classical many-sorted first-order logic with equality.
Let \(S\) be a finite set of sorts, then for each sort \(s \in S\) we let \(\VariableSymbols_{s}\) be a countably infinite set of variable symbols of the sort \(s\).
We write \(x : s\) to indicate that \(x\) is a variable symbol of sort \(s\), that is, \(x \in \VariableSymbols_{s}\).
When the sort of a variable is irrelevant or clear from the context, we omit the sort annotation and simply use the variable symbol.
We assume that the sets of variable symbols for the sorts in \(S\) are pairwise disjoint.
A many-sorted first-order language \(\Language\) over the sorts \(S\) is a set of predicate symbols of the form \(P: s_{1} \times \dots \times s_{n} \to o\) and function symbols of the form \(f: s_{1} \times \dots \times s_{n} \to s_{n+1}\), where \(P, f\) are symbols, \(s_{1}, \dots, s_{n}, s_{n + 1} \in S\) and \(o\) is a special sort symbol assumed not to appear in \(S\).
For a function symbol \(f\) the expression \(f: s_{1} \times \dots \times s_{n} \to s_{n + 1}\) with \(s_{1}, \dots, s_{n + 1} \in S\) indicates that \(f\) takes arguments of sorts \(s_1\), \dots, \(s_{n}\) to a value of sort \(s_{n + 1}\).
Similarly, for a predicate symbol \(P\) the expression of the form \(P : s_{1} \times \dots \times s_{n} \to \SortBool\) indicates that \(P\) is a predicate with arguments of sorts \(s_{1}\), \dots, \(s_{n}\).
Terms of \(\Language\) are constructed as usual from the variable symbols  and function symbols according to their respective types.
Each thus constructed term \(t\) has a uniquely determined sort \(s\) and, therefore, we call \(t\) an \(s\)-term.
Formulas of \(\Language\) are constructed from terms, predicate symbols, the connectives \(\top\), \(\bot\), \(\wedge\), \(\neg\), \(\vee\), \(\rightarrow\) and the quantifiers \(\Forall{x:s}{}\), \(\Exists{x:s}{}\) for \(s \in S\) and \(x \in \VariableSymbols_{s}\).

In this article we will make heavy, albeit elementary, use of model theoretic techniques.
Hence, we recall some basic model theoretic concepts and notations.
A first-order structure \(M\) for the language \(\Language\) (over sorts \(S\)) is a function that assigns: To each sort \(s \in S\) a non-empty set \(\DomainOf{s}{M}\); To each function symbol \(f: s_{1} \times \dots \times s_{n} \to s_{n + 1}\) a function \(f^{M}: \bigtimes_{i = 1}^{n}\DomainOf{s_{i}}{M} \to \DomainOf{s_{n + 1}}{M}\); To each predicate symbol \(P: s_{1} \times \dots \times s_{n} \to o\) a set \(P^{M} \subseteq \bigtimes_{i = 1}^{n}\DomainOf{s_{i}}{M}\).
A variable assignment \(\sigma\) is a function that assigns to each variable symbol \(v : s\) with \(s \in S\) an element of \(\DomainOf{s}{M}\).
We write \(M, \sigma \models \varphi\) if the formula \(\varphi\) is true in \(M\) under the variable assignment \(\sigma\).
Let \(\varphi(x_{1} : s_{1}, \dots, x_{n} : s_{n},\vec{y})\) be a formula and \(d_{i} \in \DomainOf{s_{i}}{M}\) for \(i = 1, \dots, n\), then we write \(M, \{ x_{i} \mapsto d_{i} \mid i = 1, \dots, n\} \models \varphi\) (or \(M \models \varphi(d_{1}, \dots, d_{n},\vec{y})\)) if \(M, \sigma \models \varphi\), for all variable assignments \(\sigma\) with \(\sigma(x_{i}) = d_{i}\) for \(i = 1, \dots, n\).
Thus, in particular, \(M \models \varphi\) if \(M,\sigma \models \varphi\) for all variable assignments \(\sigma\).
Let \(t(x_{1} : s_1, \dots, x_{n} : s_n)\) be a term and \(d_{1}, \dots, d_{n}\) a finite sequence in \( \DomainOf{s_{1}}{M}\times \cdots \times \DomainOf{s_{n}}{M}\), then we write \(t^{M}(\vec{d})\) to denote the element \(b\) of \(M\) such that \(M, \{ x_{i} \mapsto d_{i} \mid i = 1, \dots, n\} \models t = b\).

In the arguments given in \cref{sec:two_step,sec:right_cancellation} it is often necessary to consider terms and formulas of a language \(\Language\) under some partial variable assignment over an \(\Language\) structure \(M\).
In order to simplify the notation, we let \(\Language(M)\) denote the language \(\Language\) extended by a fresh function symbol \(c_{d} : s\) for each element \(d \in \DomainOf{s}{M}\) and sort \(s \in S\).
Moreover, we let the structure \(M\) interpret the language \(\Language(M)\) by letting \(M\) interpret \(c_{d}\) as the object \(d\).

In this article we define a theory \(T\) to be a set of sentences, which we call the axioms of \(T\).
Let \(\varphi\) be formula, then we write \(T \vdash \varphi\) if \(\varphi\) is provable in (many-sorted) first-order logic from the axioms of \(T\).
Let \(T_{1}, T_{2}\) be theories, then \(T_{1} + T_{2}\) denotes the theory axiomatized by the set of sentences \(T_{1} \cup T_{2}\).

Finally, let us define some notation for some particular sets of formulas.
By \(\Open(\Language)\) we denote the set of quantifier-free formulas of the language \(\Language\).
Let \(\SetOfFormulas\) be a set of formulas, then we write \(\forall_{1}(\SetOfFormulas)\) (\(\exists_{1}(\SetOfFormulas)\)) for the set of formulas in \(\SetOfFormulas\) of the form \(\Forall{\vec{x}}{\varphi}\) (\(\Exists{\vec{x}}{\varphi}\)) where \(\varphi\) is a quantifier-free formula and \(\vec{x}\) is a possibly empty sequence of variables.
We also write \(\forall_{1}(\Language)\) for the formulas of the above form in the language \(\Language\).
\subsection{Induction and lists}
\label{sec:preliminaries:induction}
In this section we introduce the basic construction of finite Lisp-like lists that we work with in this article.
We also recall the traditional induction schema for lists and its related terminology.
Throughout the article we will consider various forms of induction that will be defined when needed.
We use the traditional induction schema as defined in this section as a reference in the sense that we justify the other induction schemata in terms of the traditional one.

Now we will define the basic language of finite Lisp-like lists and the corresponding induction schema.
\begin{definition}
  \label{def:1}
  The language \(\LanguageListBase\) consists of the sort \(\SortIndividuals\) of elements and the sort \(\SortList\) of finite lists.
  Moreover, the language \(\LanguageListBase\) contains the function symbols \(\Nil : \SortList\) and \(\ConsSymbol : \SortIndividuals \times \SortList \to \SortList\).
\end{definition}
Informally, the symbol \(\Nil\) denotes the empty list and \(\ConsSymbol\) denotes the operation that adds a given element to the front of a given list.
For the sake of legibility we will use upper case letters \(X\), \(Y\), \(Z\) and variants thereof to denote variables that range over the sort \(\SortList\).
For these variables we omit the the sort annotation, that is, the ``: \(\SortList\)'' part.

The traditional induction schema for Lisp-like lists is analogous to the one for natural numbers with the exception that the induction step also quantifies over elements.
\begin{definition}
  \label{def:induction_schema_list}
  Let \(\varphi(X,\vec{z})\) be a formula, then the formula \(\ListInductionAxiom{X}{\varphi}\) is given by
  \[
    \left( \varphi(\Nil,\vec{z}) \wedge \Forall{X}{\Forall{x}{\left( \varphi(X, \vec{z}) \rightarrow \varphi(\Cons{x}{X}, \vec{z})) \right)}} \right) \rightarrow \Forall{X}{\varphi(X,\vec{z})}.
  \]
  For a set of formulas \(\SetOfFormulas\), the theory \(\ListInductionSchema{\SetOfFormulas}\) is axiomatized by the universal closure of the formulas \(\ListInductionAxiom{X}{\varphi}\), where \(\varphi(X,\vec{z}) \in \SetOfFormulas\).
\end{definition}
The induction schema given above is parameterized by the set of possible induction formulas.
This permits to consider various theories by varying the structure of the induction formulas.

We will also refer to the above induction principle as one-step induction, since the induction step proceeds by a step of size one.
In \cref{sec:two_step} we will introduce the big-step induction principle that proceeds in larger steps. 

When we work with theories of lists we usually work over the following base theory that provides the disjointness and the injectivity of the list constructors \(\Nil\) and \(\ConsSymbol\).
\begin{definition}
  \label{def:theory_list_base}
  The theory \(\TheoryListBase\) is axiomatized by the following axioms
  \begin{gather}
    \Nil \neq \Cons{x}{X}, \tag{L0.1} \label{ax:thlistbase:1}
    \\
    \Cons{x}{X} = \Cons{y}{Y} \rightarrow x = y \wedge X = Y. \tag{L0.2} \label{ax:thlistbase:2}
  \end{gather}
\end{definition}
\subsection{Transfinite sequences}
\label{sec:preliminaries:sequences}
In this section we introduce some notations and definitions related to transfinite sequences, that is, sequences indexed by ordinals.
Later on in \cref{sec:two_step,sec:right_cancellation}, we will heavily rely on transfinite sequences, of length up to \(\omega^{3}\), for the construction of non-standard models of induction over lists.

Let \(\mathcal{X}\) be a set and \(\alpha\) be an ordinal number, then as usual \(\mathcal{X}^{\alpha}\) denotes the set \(\{ f: \alpha \to \mathcal{X} \}\) of sequences of elements of \(\mathcal{X}\) with length \(\alpha\).
By \(\mathcal{X}^{<\alpha}\) (\(\mathcal{X}^{\leq \alpha}\)) we denote the set \(\bigcup_{\beta < \alpha}\mathcal{X}^{\beta}\) (\(\bigcup_{\beta \leq \alpha}\mathcal{X}^{\beta}\)).
In particular, we denote by \(\mathcal{X}^{*}\) the set \(\mathcal{X}^{<\omega} = \bigcup_{i \in \NaturalNumbers}\mathcal{X}^{i}\) of all finite sequences of elements of \(\mathcal{X}\).
Let \(a \in \mathcal{X}^{\leq \alpha}\), then \(|a|\) denotes the ordinal \(\beta \leq \alpha\) such that \(a \in \mathcal{X}^{\beta}\).
The empty sequence \(()\) (\(= \varnothing\)) is also denoted by \(\varepsilon\) and \((x)\) denotes the one element sequence \(\{ 0 \mapsto x \}\).

In the following we define concatenation of ordinal indexed sequences.
The definition as given below relies on the well-definedness of ordinal subtraction of an ordinal \(\beta\) from an ordinal \(\alpha\) when \(\alpha \geq \beta\) (see \cite[Theorem~8.8]{takeuti1971}).
\begin{definition}
  \label{def:concatenation_ord_indexed_sequences}
  Let \(\mathcal{X}\) be a set, \(\alpha, \beta\) ordinals, \(a \in \mathcal{X}^{\alpha}\), and \(b \in \mathcal{X}^{\beta}\), then the sequence \(a \frown b \in \mathcal{X}^{\alpha + \beta}\) is defined by
  \begin{gather*}
  (a \frown b)_{\gamma} \coloneqq
 \begin{cases}
   a_{\gamma} & \text{if \(\gamma < \alpha\)}
   \\
   b_{\delta} & \text{otherwise}
 \end{cases},
\end{gather*}
where \(\gamma < \alpha + \beta\) and \(\delta\) is the unique ordinal such that \(\alpha + \delta = \gamma\).
\end{definition}
Observe that the definition of concatenation of ordinal indexed sequences given above generalizes the concatenation of finite sequences, since
\begin{gather*}
  (a_{0}, \dots, a_{n-1}) \frown (b_{0}, \dots, b_{m-1}) = (a_{0}, \dots, a_{n-1}, b_{0}, \dots, b_{m- 1}).
\end{gather*}
The concatenation of ordinal indexed sequences as defined above has some interesting properties.
\begin{lemma}
  \label{lem:ordinal_indexed_sequence_concatenation_has_left_cancellation}
  Let \(\mathcal{X}\) be a set, \(\alpha, \beta, \gamma\) ordinals, \(a \in \mathcal{X}^{\alpha}\), \(b \in \mathcal{X}^{\beta}\), \(c \in \mathcal{X}^{\gamma}\), then we have:
  \begin{enumerate}[label=(\roman*)]
  \item associativity: \(a \frown (b \frown c) = (a \frown b) \frown c\);
  \item left cancellation: If \(a \frown b = a \frown c\), then \(b = c\).
  \end{enumerate}
\end{lemma}
\begin{proof}
  For (\textit{i}) let \(\mu < \alpha + \beta + \delta\).
  By the associativity of ordinal addition we have \(a \frown (b \frown c), (a\frown b) \frown c \in \mathcal{X}^{\alpha+\beta+\gamma}\).
  Now we have to consider three cases.
  If \(\mu < \alpha\), then \(\mu < \alpha + \beta\).
  Hence \[
    (a \frown (b \frown c))_{\mu} = a_{\mu} = (a \frown b)_{\mu} = ((a \frown b) \frown c)_{\mu}.
  \]
  If \(\alpha \leq \mu < \alpha + \beta\), then there is a unique ordinal \(\delta\) such that \(\alpha + \delta = \mu\).
  Moreover, by the monotonicity properties of ordinal addition we have \(\delta < \beta\).
  Hence
  \[
    (a \frown (b \frown c))_{\mu} = (b \frown c)_{\delta} = b_{\delta} = (a \frown b)_{\mu} = ((a \frown b) \frown c)_{\mu}.
  \]
  Finally, if \(\alpha + \beta \leq \mu < \alpha + \beta + \gamma\), then there are unique \(\delta_{1}\) and \(\delta_{2}\) such that \(\alpha + \delta_{1} = \mu\) and \(\alpha + \beta + \delta_{2} = \mu\).
  Furthermore, we have \(\delta_{1} \geq \beta\), hence there is \(\delta_{3}\) such that \(\beta + \delta_{3} = \delta_{1}\).
  Thus \(\mu = \alpha + \delta_{1} = \alpha + \beta + \delta_{3} = \alpha + \beta + \delta_{2}\), hence \(\delta_{2} = \delta_{3}\).
  Therefore
  \[
    (a \frown (b \frown c))_{\mu} = (b \frown c)_{\beta + \delta_{3}} = c_{\delta_{3}} = c_{\delta_{2}} = ((a \frown b) \frown c)_{\mu}.
  \]
  
  For (\textit{ii}) observe that since \(a \frown b \in \mathcal{X}^{\alpha + \beta}\), \(a \frown c \in \mathcal{X}^{\alpha + \gamma}\) we have \(\alpha + \beta = \alpha + \gamma\) and therefore by the left cancellation of ordinal addition \(\beta = \gamma\).
  Now let \(\delta < \beta\), then we have \((a \frown b)_{\alpha + \delta} = b_{\delta} = c_{\delta} = (a \frown c)_{\alpha + \delta}\).
  Hence, \(b = c\).
\end{proof}
Observe, however, that since already ordinal addition does not have right cancellation, the concatenation of ordinal indexed sequences does clearly also not have right cancellation.

For sequences we will often be interested in suffixes.
In the following definition we introduce some notation for accessing the suffix of a sequence.
\begin{definition}[Sequence suffix]
  \label{def:suffix}
  Let \(\mathcal{X}\) be a set, \(\alpha, \beta\) ordinals with \(\beta \leq \alpha\), and \(a \in \mathcal{X}^{\alpha}\), then the sequence \(a \uparrow \beta\) is given by
  \begin{equation*}
    (a \uparrow \beta)_{\gamma} = a_{\beta + \gamma},
  \end{equation*}
  for \(\gamma < \mu\) where \(\mu\) is the unique ordinal such that \(\beta + \mu = \alpha\).
\end{definition}
Finally, let us give some notation for the sequence obtained by concatenating sequences of uniform length.
This construction will be used in \cref{sec:right_cancellation} and relies on ordinal division with remainder (see \cite[Theorem~8.27]{takeuti1971}). 
\newcommand{\flatten}[1]{\lfloor #1 \rfloor}
\begin{definition}
  \label{def:transfinite_sequence_flattening}
  Let \(\mathcal{X}\) be a set, \(\alpha, \beta\) ordinals with \(\alpha > 0\), and \(\mathfrak{a} \in (\mathcal{X}^{\alpha})^{\beta}\).
  The sequence \(\flatten{\mathfrak{a}} \in \mathcal{X}^{\alpha\cdot\beta}\) is defined by
  \[
    \flatten{\mathfrak{a}}_{\xi} \coloneqq \mathfrak{a}_{\delta,\mu},
  \]
  for \(\xi < \alpha \cdot \beta\) where \(\mu, \delta\) are the unique ordinals such that \(\xi = (\alpha \cdot \delta) + \mu \) with \(\mu < \alpha\).
  Furthermore, for \(a \in \mathcal{X}^{\alpha}\), we denote by \(a^{\beta}\) the sequence \(\flatten{(a)_{\gamma<\beta}}\) consisting of \(\beta\) times the sequence \(a\).
\end{definition}

\section{Big-step induction}
\label{sec:two_step}
Big-step induction is a generalization of the induction principle of \cref{def:induction_schema_list} in which the induction step proceeds by adding more than one element.
Big-step induction and other induction principles are often used in automated inductive theorem provers \cite{bundy2005rippling}.
Some formulas can be proved more naturally by a special induction principle.
Hence a special induction principle may allow a prover to find a proof faster under the constraints of its proof search algorithm or even enable the prover to prove the formula in the first place \cite{vierling2022}.
It is therefore interesting to investigate the relation between the one-step induction principle and special induction principles implemented in \ac{AITP} systems.

In this section we show the first main result of this article, namely that, for all \(m \geq 1\), quantifier-free \((m + 1)\)-step induction for lists does not follow from quantifier-free \(m\)-step induction.
In particular, quantifier-free big-step induction for lists cannot be reduced to quantifier-free one-step induction, which is in contrast to induction on natural numbers where such a reduction is possible (see for example \cite{vierling2022}).

The definition below defines the big-step induction principle for lists considered in this article.
Let us introduce some notation to make it easier to state big-step induction for lists.
Let \(t_{1}, \dots, t_{n}\) be a possibly empty list of terms of sort \(\SortIndividuals\) and \(T\) a term of sort \(\SortList\), then the term \(\ConsMult{t_{1}, \dots, t_{n}}{T}\) is defined inductively by
\begin{gather*}
  \ConsMult{}{T} = T, \\
  \ConsMult{t_{1}, \dots, t_{n + 1}}{T} = \Cons{t_{1},\dots,t_{n}}{\Cons{t_{n+1}}{T}}.
\end{gather*}
\begin{definition}
  \label{def:big_step_list_induction}
  Let \(\Formula(x,\vec{z})\) be a formula and \(m \geq 1\), then the formula \(\ListBigStepInductionAxiom{x}{m}{\Formula}\) is given by
  \begin{multline*}
    \left(
    \begin{split}
      & \bigwedge_{i = 1, \dots, m} \Forall{x_1,\dots,x_{i-1}} \Formula(\ConsMult{x_{1}, \dots, x_{i-1}}{\Nil},\vec{z})
      \\
      & \wedge \Forall{X}{\Forall{x_{1}, \dots, x_{m}}{\left(\Formula(X,\vec{z}) \rightarrow \Formula(\ConsMult{x_{1},\dots,x_{m}}{X},\vec{z})\right)}}
    \end{split}
  \right) \rightarrow \Forall{X}{\Formula(X,\vec{z})}.
  \end{multline*}
  Let \(\SetOfFormulas\) be a set of formulas and \(m \geq 1\), then the \(m\)-step induction schema \(\ListBigStepInductionSchema{\SetOfFormulas}{m}\) over \(\SetOfFormulas\) is axiomatized by the universal closure of the formulas \(\ListBigStepInductionAxiom{x}{m}{\Formula}\) where \(\Formula(x,\vec{z}) \in \SetOfFormulas\).
\end{definition}
A simple example of formulas that have natural proofs by big-step induction are the acyclicity formulas given below, which express that adding a finite number \(n \geq 1\) of elements to a list results in a different list:
\[
  X \neq \ConsMult{x_{1},\dots,x_{n}}{X}. \label{assumption:1} \tag{\(\star\)}
\]
To prove this formula by \(n\)-step induction it suffices to proceed by induction on \(X\)
 in the formula itself.
For the base case we have to show that \(\Nil \neq \ConsMult{x_{1},\dots,x_{n}}{\Nil}\), which follows readily from \eqref{ax:thlistbase:1}.
For the induction step we assume \eqref{assumption:1}.
For a contradiction assume \[
  \ConsMult{x_{1}',\dots,x_{n}'}{X} = \ConsMult{x_{1},\dots,x_{n},x_{1}',\dots,x_{n}'}{X}.
\]
Then by an \(n\)-fold application of \eqref{ax:thlistbase:2} we obtain \(x_{i}' = x_{i}\) for \(i = 1, \dots, n\) and
\[
  X = \ConsMult{x_{1}',\dots,x_{n}'}{X}.
\]
This contradicts the induction hypothesis and thus completes the induction step.
Interestingly, however, the acyclicity formula \eqref{assumption:1} also has a slightly less natural proof using one-step quantifier-free induction.
\begin{lemma}
  \label{lem:11}
  The theory \(\TheoryListBase + \ListInductionSchema{\Open(\LanguageListBase)}\) proves
  \begin{enumerate}[label=(\roman*)]
  \item \(X \neq \ConsMult{x_{1},\dots,x_{n}}{X}\) for \(n \geq 1\);
  \item \(X = \Nil \vee \Exists{x'}{\Exists{X'}{X = \Cons{x'}{X'}}}\).
  \end{enumerate}  
\end{lemma}
\begin{proof}
  For (\textit{i}) we assume \(X = \ConsMult{x_{1},\dots,x_{n}}{X}\) and proceed by induction on \(X'\) in the formula
  \[
    \underbrace{X' \neq X}_{\psi_{1}(X')} \wedge \underbrace{X' \neq \Cons{x_{n}}{X}}_{\psi_{2}(X')} \wedge \dots \wedge \underbrace{X' \neq \ConsMult{x_{2},\dots,x_{n}}{X}}_{\psi_{n}(X')}.
  \]
  For the base case we have to show \(\psi_{i}(\Nil)\) for \(i = 1, \dots, n\).
  For \(i > 1\), this follows easily from \eqref{ax:thlistbase:1} and for \(i = 0\) we obtain \(X' \neq X\) from the assumption \(X = \ConsMult{x_{1},\dots,x_{n}}{X}\) and \eqref{ax:thlistbase:1}.
  For the induction step we assume \(\bigwedge_{i = 1}^{n}\psi_{i}(X')\).
  Let \(i \in \{ 1, \dots, n\}\).
  If \(i = 1\) assume \(\Cons{x'}{X'} = X\), then by the assumption \(X = \ConsMult{x_{1}, \dots, x_{n}}{X}\) and \(\eqref{ax:thlistbase:2}\) we obtain \(X' = \ConsMult{x_{2},\dots,x_{n}}{X}\) which contradicts the assumption \(\psi_{n}(X')\).
  If \(i > 1\), then assume \(\Cons{x'}{X'} = \ConsMult{x_{n - i + 2},\dots,x_{n}}{X}\), then by \eqref{ax:thlistbase:2} we obtain \(X' = \ConsMult{x_{n - i + 1},\dots,x_{n}}{X}\), which contradicts \(\psi_{i - 1}(X')\).
  Hence, we finally obtain \(\Forall{X'}{\left(\bigwedge_{i = 1}^{n}\psi_{i}(X')\right)}\).
  Thus, in particular, we have \(\bigwedge_{i = 1}\psi_{i}(\ConsMult{x_{1},\dots,x_{n}}{X})\).
  Therefore, we obtain \[\ConsMult{x_{1},\dots,x_{n}}{X} \neq X,\] which contradicts the first assumption.

  For (\textit{ii}) we proceed by induction on \(Y\) in the formula \(X \neq Y\).
  For the base case we have to show \(X \neq \Nil\).
  Assume \(X = \Nil\), then we are done.
  For the induction step, we assume \(X \neq Y\) and \(X = \Cons{y}{Y}\) and we are done.
  Hence we have \(\Forall{Y}{X \neq Y}\).
  Thus in particular \(X \neq X\), which is a contradiction and thus implies the claim.
\end{proof}
This gives rise to the question whether a similar technique as we have used to prove the acyclicity formulas with quantifier-free induction is also possible in general.
It is straightforward to see that we can simulate big-step induction with single-step induction by making use of universal quantifiers and conjunction.
For the sake of completeness we recall the argument.
\begin{lemma}
  \label{lem:12}
  Let \(m \geq 1\) and \(\varphi(X,\vec{z})\) be a formula, then
  \[
    \vdash \ListInductionAxiom{X}{\bigwedge_{i = 1}^{m}\Forall{x_{1}}{\dots\Forall{x_{i-1}}{\varphi(\ConsMult{x_{1},\dots,x_{i-1}}{X},\vec{z})}}} \rightarrow \ListBigStepInductionAxiom{X}{m}{\varphi(X,\vec{z})}.
  \]
\end{lemma}
\begin{proof}
  Assume \(\Forall{x_{1}}{\dots\Forall{x_{i-1}}{\varphi(\ConsMult{x_{1},\dots,x_{i-1}}{\Nil},\vec{z})}}\) for \(i = 1, \dots, m\) and
  \begin{equation*}
    \label{assumption_2}
    \Forall{X}{\Forall{x_{1}}{\dots\Forall{x_{m}}{\left(\varphi(X,\vec{z}) \rightarrow \varphi(\ConsMult{x_{1},\dots,x_{m}}{X},\vec{z})\right)}}}. \tag{\(\star\)}
  \end{equation*}
  Clearly it suffices to show the following formula.
  \begin{equation*}
    \label{induction_formula}
    \bigwedge_{j = 1}^{m}\Forall{x_{1}}{\dots\Forall{x_{j-1}}{\varphi(\ConsMult{x_{1},\dots,x_{j-1}}{X},\vec{z})}}. \tag{\(\dagger\)}
  \end{equation*}
  We proceed by induction on \(X\) in the formula \eqref{induction_formula}.
  The base case follows immediately from the assumptions.
  For the induction step case we assume \eqref{induction_formula}.
  Now let \(i \in \{ 1, \dots, m \}\), let \(x', x_{1}, \dots, x_{i -1}\) be fixed but arbitrary.
  If \(i < m\), then we have to show \(\varphi(\ConsMult{x_{1},\dots,x_{i-1},x'}{X},\vec{z})\), which follows from the induction hypothesis with \(j = i + 1\).
  If \(i = m\), then we have to show \(\varphi(\ConsMult{x_{1},\dots,x_{m-1},x'}{X},\vec{z})\).
  By \eqref{induction_formula} with \(j = 1\), we have \(\varphi(X,\vec{z})\), hence by \eqref{assumption_2} we obtain the desired formula.
\end{proof}
\begin{remark}
  \label{rem:1}
  When the domain of the elements provably consists of a finite number of elements \(n \geq 1\), then the quantifiers over elements in the induction formula of \cref{lem:12} can be replaced by a conjunction.
  Hence, in this situation \((m+1)\)-step induction reduces to \(m\)-step induction without an increase in quantifier-complexity of the induction formulas.
\end{remark}
However, as we will show, the increase of the quantifier complexity when simulating big-step induction with one-step induction is in general unavoidable.
The remainder of the section is devoted to the proof of the following proposition.
\begin{definition}
  \label{def:language_list_base_with_unary_predicate}
  The language \(\LanguageListA\) extends the base language of lists \(\LanguageListBase\) by the predicate symbol \(A : \SortList \to \SortBool\).
\end{definition}
\begin{proposition}
  \label{pro:two_step:main_claim}
  Let \(m \geq 2\), then \[
    \TheoryListBase + \bigcup_{1\leq j < m}\ListBigStepInductionSchema{\Open(\LanguageListA)}{j} \not \vdash \ListBigStepInductionAxiom{x}{m}{A(x)}.
  \]
\end{proposition}
This proposition entails, in particular, that \( \TheoryListBase + \ListInductionSchema{\Open(\LanguageListA)} \not \vdash \TwoStepInductionSchema{\Open(\LanguageListA)} \).
\newcommand{\StructureListOneStep}[1]{\FOStructure_{1}^{#1}}
We will show the above claim by constructing a model of quantifier-free induction over the language \(\LanguageListA\) in which the predicate \(A\) does not satisfy two-step induction.
In such a model we call an element a standard element if it can be expressed as a term of the form \(\Cons{x_{1}}{\dots\Cons{x_{n}}{\Nil}}\) under a suitable variable assignment.
All other elements are called the non-standard elements.
By \cref{lem:11}.(\textit{ii}) a non-standard element can be decomposed any finite number of times and thus resemble transfinite sequences of length at least \(\omega\).
The model constructed in \cref{def:two_step:non_standard_model} will use transfinite sequences of length up to \(\omega\) as the non-standard elements.
Since for example the transfinite sequence \(v^{\omega}\) with \(v \in \NaturalNumbers^{*}\) satisfies \(v^{\omega} = v \frown v^{\omega}\) it violates the acyclicity property \(X \neq \Cons{x_{1},\dots,x_{|v|}}{X}\) (see \cref{lem:11}).
Hence we have to avoid sequences that absorb a finite prefix.

The following definition introduces the non-standard elements that we use for the model
 constructed in this section.
\begin{definition}
  \label{def:two_step:non_standard_elements}
  Let \(k \in \NaturalNumbers\), then by \(N_{k}\) we denote the sequence \((i)_{k \leq i < \omega}\). Now we define
  \[
    \mathcal{N} \coloneqq \{ w \frown N_{k} \mid w \in \NaturalNumbers^{*}, k \in \NaturalNumbers\}.
  \]
  Let \(N \in \mathcal{N}\), then there is a unique decomposition \(N = w \frown N_{k}\) such that \(|w|\) and \(k\) are minimal.
  We write \(w_{N}\) for this \(w\) and \(k_{N}\) for this \(k\).
  We call \(w_{N}\) the main prefix of \(N\) and \(N_{k_{N}}\) the main suffix of \(N\).
\end{definition}
We can now define a structure whose domain consists of the finite sequences of natural numbers and the non-standard elements defined above.
\newcommand{\divides}[2]{#1 \mid #2}
\begin{definition}
  \label{def:two_step:non_standard_model}
  Let \(m \geq 1\), then the structure \(\StructureListOneStep{m}\) interprets the sort \(\SortIndividuals\) as the natural numbers and the sort \(\SortList\) as
  \( \StructureListOneStep{m}(\SortList) = \NaturalNumbers^{*} \cup \mathcal{N} \).
  Furthermore, \(\StructureListOneStep{m}\) interprets the non-logical symbols as follows
  \begin{gather*}
    \Nil^{\StructureListOneStep{m}} \coloneqq \varepsilon,
    \\
    \ConsSymbol^{\StructureListOneStep{m}}(n,l) \coloneqq (n) \frown l,
    \\
    A^{\StructureListOneStep{m}} \coloneqq \NaturalNumbers^{*} \cup \left\{ N \in \mathcal{N} \mid \text{\(w_{N} \neq \varepsilon\) or \(\notdivides{m}{k_{N}}\)}\right\}.
  \end{gather*}
\end{definition}
We say that an element \(l_{1}\) is a predecessor of an element \(l_{2}\) if there are \(k,n_{1},\dots,n_{k} \in \NaturalNumbers\) such that \(\StructureListOneStep{m} \models l_{2} = \ConsMult{n_{1},\dots,n_{k}}{l_{1}}\).

We start by observing that the structure defined above satisfies the basic axioms of the constructors \(\Nil\) and \(\ConsSymbol\) of finite sequences.
\begin{lemma}
  \label{lem:two_step:structure_satisfies_base_theory}
  Let \(m \geq 1\), then
  \(\StructureListOneStep{m} \models \TheoryListBase\).
\end{lemma}
\begin{proof}
  We start with the axiom \eqref{ax:thlistbase:1}.
  We have \(\Nil^{\StructureListOneStep{m}} = \varepsilon = \varnothing\).
  Let \(n \in \NaturalNumbers\) and \(l \in \StructureListOneStep{m}(\SortList)\), then \((0,n) \in (n) \frown l = \ConsSymbol^{\StructureListOneStep{m}}(n,l)\).
  Hence, \(\Nil^{\StructureListOneStep{m}} \neq \ConsSymbol^{\StructureListOneStep{m}}(n,l)\) and therefore \(\StructureListOneStep{m} \models \eqref{ax:thlistbase:1}\).
  Now let \(n_{1}, n_{2} \in \NaturalNumbers\) and \(l_{1}, l_{2} \in \StructureListOneStep{m}(\SortList)\) and assume that \(\ConsSymbol^{\StructureListOneStep{m}}(n_{1},l_{1}) = (n_{1}) \frown l_{1} = (n_{2}) \frown l_{2}\).
  Clearly, \(n_{1} = n_{2}\), hence by \cref{lem:ordinal_indexed_sequence_concatenation_has_left_cancellation} we immediately obtain \(\StructureListOneStep{m} \models \eqref{ax:thlistbase:2}\).
\end{proof}
The following lemma shows that unary \(A\) predicates of \(\StructureListOneStep{m}\), eventually periodically become true on predecessors of non-standard elements.
\begin{lemma}
  \label{lem:A_atoms_eventually_true_on_even_main_suffixes}
  Let \(m \geq 1\), \(t(X)\) be a \(\LanguageListA(\StructureListOneStep{m})\) term, then there is a \(K \in \NaturalNumbers\) such that for all \(k \in \NaturalNumbers\) with \(k \geq K\) and \(\notdivides{m}{k}\),
  \(
    \StructureListOneStep{m} \models A(t(N_{k})).
  \)
\end{lemma}
\begin{proof}
  There clearly is a \(w \in \NaturalNumbers^{*}\) such that \(\StructureListOneStep{m} \models t(l) = w \frown l\) for all \(l \in \StructureListOneStep{m}(\SortList)\).
  If \(w = \varepsilon\), then we are done by letting \(K = 0\).
  Otherwise, we let \(K = (w)_{|w| - 1} + 2\).
  For \(k \geq K\), the sequence \(N_{k}\) is the main suffix of the sequence \(w \frown N_{k}\) and the main prefix of \(w \frown N_{k}\) is not empty. Thus \(\StructureListOneStep{m} \models A(t(N_{k}))\).
\end{proof}
Informally, the following lemma states that unary equational predicates over elements of \(\StructureListOneStep{m}\) eventually stabilize.
\begin{lemma}
  \label{lem:list_equations_stabilize}
  Let \(m \geq 1\) and \(E(X)\) an \(\LanguageListA(\StructureListOneStep{m})\) equation. If \(\StructureListOneStep{m} \not \models E(X)\) then there exists \(K \in \NaturalNumbers\) such that firstly \(\StructureListOneStep{m} \not \models E(w)\) for all \(w \in \NaturalNumbers^{*}\) with \(|w| \geq K\) and secondly \(\StructureListOneStep{m} \not \models E(N_{k})\) for all \(k \geq K\).
\end{lemma}
\begin{proof}
  The case where \(X \notin \Var(E)\) is trivial.
  Let \(E(X)\) be \(u(X) = v(X)\).
  If \(X \in \Var(u)\) and \(\Var(v) = \varnothing\), then there is \(w \in \NaturalNumbers^{*}\) and \(l' \in \StructureListOneStep{m}(\SortList)\) such that \(\StructureListOneStep{m} \models u(l) = w \frown l\) for all \(l \in \StructureListOneStep{m}(\SortList)\) and \(\StructureListOneStep{m} \models v = l'\).
  If \(|w| > |l'|\), then we are done by letting \(K = 0\).
  Otherwise, if \(|w| \leq |l'|\) we consider the prefix of \(l'\).
  If \(w\) is not a prefix of \(l'\), then again we are done by letting \(K = 0\).
  If \(w\) is the prefix of \(l'\), then \(l' = w \frown l^{\prime\prime}\) for some \(l^{\prime\prime} \in \NaturalNumbers^{\leq |l'|}\).
  We have \(l^{\prime\prime} \in \StructureListOneStep{m}(\SortList)\), since \(\StructureListOneStep{m}(\SortList)\) is closed under predecessors.
  Thus \(\StructureListOneStep{m} \models E(l)\) if and only if \(l = l^{\prime\prime}\).
  If \(l^{\prime\prime}\) is a standard element, then \(\StructureListOneStep{m} \not \models E(l)\) for all \(l \in \StructureListOneStep{m}(\SortList)\) with \(|l| > |l^{\prime\prime}|\).
  Hence, we let \(K = |l^{\prime\prime}| + 1\).
  Otherwise, if \(l^{\prime\prime}\) is non-standard, then we readily have \(\StructureListOneStep{m} \not \models E(l)\) for all \(l \in \NaturalNumbers^{*}\).
  Furthermore, we have \(\StructureListOneStep{m} \not \models E(l)\) for all non-standard \(l \in \StructureListOneStep{m}(\SortList)\) with \((l)_{0} \neq (l^{\prime\prime})_{0}\).
  Hence, it suffices to let \(K = (l^{\prime\prime})_{0} + 1\).

  Now let us consider the case where \(\Var(u) \cap \Var(v) = \{ X \}\).
  There exist \(w, w' \in \NaturalNumbers^{*}\) such that for all \(l \in \StructureListOneStep{m}(\SortList)\), \(\StructureListOneStep{m} \models u(l) = w \frown l\) and \(\StructureListOneStep{m} \models v(l) = w' \frown l\).
  Moreover, by the assumption that \(\StructureListOneStep{m} \not \models E(X)\) we have \(w \neq w'\).
  Hence, \(\StructureListOneStep{m} \not \models E(l)\) for all \(l \in \StructureListOneStep{m}(\SortList)\).
  Thus, we let \(K = 0\).
\end{proof}
We are now ready to show that the structure \(\StructureListOneStep{m}\) satisfies quantifier-free \(j\)-step induction for \(1 \leq j < m\) over the language consisting of the list constructors \(\Nil\), \(\ConsSymbol\), and the predicate symbol \(A\).
\begin{lemma}
  \label{lem:two_step:structure_satisfies_open_induction}
  Let \(m \geq 2\), then \(\StructureListOneStep{m} \models \bigcup_{1 \leq j < m }\ListBigStepInductionSchema{\Open(\LanguageListA)}{j}\).
\end{lemma}
\begin{proof}
  Let \(j \in \NaturalNumbers\) with \(1 \leq j < m\) and \(\Formula(X)\) be a quantifier-free \(\LanguageListA(\StructureListOneStep{m})\) formula.
  Assume that
  \begin{gather}
    \StructureListOneStep{m} \models \Formula(\ConsMult{x_{1}, \dots, x_{i-1}}{\Nil}),
    \tag{\(\ast\)} \label{assumption:IB}
  \end{gather}
  for \(i = 1, \dots, j\) and
  \begin{gather}
    \StructureListOneStep{m} \models \Formula(X) \rightarrow \Formula(\ConsMult{x_{1},\dots,x_{j}}{X}),
    \tag{\(\star\)} \label{assumption:IS}
  \end{gather}
  Let \(l \in \StructureListOneStep{m}(\SortList)\).
  We have to show that \(\StructureListOneStep{m} \models \Formula(l)\).
  If \(l\) is standard, then we are done by a straightforward induction on \(|l|\) making use of \eqref{assumption:IB} and \eqref{assumption:IS}.
  
  Now let us consider the case where \(l\) is non-standard, that is, \(l \in \mathcal{N}\).
  Let \(E_{1}(X), \dots, E_{n}(X)\) be all the list equations of \(\Formula\) with \(\StructureListOneStep{m} \not \models E_{i}(X)\) for \(i = 1, \dots, n\).
  Then by \cref{lem:list_equations_stabilize} there exists \(K \in \NaturalNumbers\) such that \(\StructureListOneStep{m} \not \models E_{i}(w)\) for all \(w \in \NaturalNumbers^{*}\) with \(|w| \geq K\) and \(\StructureListOneStep{m} \not \models E_{i}(N_{k})\) for all \(k \geq K\).
  
  Now let \(A(t_{1}(X))\), \dots, \(A(t_{p}(X))\) be all the \(A\) atoms of \(\Formula\).
  By \cref{lem:A_atoms_eventually_true_on_even_main_suffixes}, there exists \(K' \geq K\) such that for all \(k \in \NaturalNumbers\) with \(k \geq K'\) and \(\notdivides{m}{k}\), we have \(M \models A(t_{q}(N_{k}))\) for \(q = 1, \dots, p\).
  Hence, by taking a sufficiently long prefix \(w\) of \(l\) (\(|w| \geq j\)), we obtain \(K^{\prime\prime} \geq K^{\prime}\) such that \(l = w \frown N_{K^{\prime\prime}}\) and \(\divides{m}{K^{\prime\prime} - 1}\).
Since \(\divides{m}{K^{\prime\prime} -1}\) and \(j < m\), we have \(\notdivides{m}{K^{\prime\prime} + i}\) for \(i = 0, \dots, j-1\).
Thus, \(\StructureListOneStep{m} \models A(t_{q}(l \uparrow |w| - i))\) for \(q = 1, \dots, p\) and \(i = 0, \dots, j -1\).

  Let \(\FormulaB(X)\) be any atom of \(\Formula(X)\) and  \(w' \in \NaturalNumbers^{*}\) with \(|w'| \geq K\), then by the above, for \(i = 0, \dots, j -1\), we have \(\StructureListOneStep{m} \models \FormulaB(w')\) if and only if \(\StructureListOneStep{m} \models \FormulaB(l \uparrow |w| - i)\).
  Hence, \(\StructureListOneStep{m} \models \Formula(l \uparrow |w| - i) \leftrightarrow \Formula(w')\).
  In the first part of the proof we have already shown that \(\StructureListOneStep{m} \models \Formula(w')\).
  Hence, we have \(\StructureListOneStep{m} \models \Formula(l \uparrow |w| - i)\).
  
  Therefore, by a straightforward induction starting with \(\StructureListOneStep{m} \models \Formula(l \uparrow |w| - i)\) for \(i = 0, \dots, j -1\) and by making use of \eqref{assumption:IS} we obtain \(\StructureListOneStep{m} \models \Formula( w' \frown (l \uparrow |w| ))\) for all \(w' \in \NaturalNumbers^{*}\).
  In particular, we have \(\StructureListOneStep{m} \models \Formula(l)\).
\end{proof}
\begin{lemma}
  \label{lem:two_step_induction_counterexample}
  Let \(m \geq 2\), then \(\StructureListOneStep{m} \not \models \ListBigStepInductionAxiom{x}{m}{A(x)}\).
\end{lemma}
\begin{proof}
  We have \(\StructureListOneStep{m} \models A(w)\) for all \(w \in \NaturalNumbers^{*}\), hence in particular
  \[
    \text{\(\StructureListOneStep{m} \models A(\ConsMult{x_{1},\dots,x_{j-1}}{\Nil})\) for \(j = 1, \dots, m\).}
  \]
  Now we consider the induction step.
  Let \(l \in \StructureListOneStep{m}(\SortList)\) and \(n_{1}, \dots, n_{m} \in \NaturalNumbers\).
  If \(l \in \NaturalNumbers^{*}\), then by the above we have
  \[
    (\ConsMult{n_{1},\dots,n_{m}}{l})^{\StructureListOneStep{m}} \in \NaturalNumbers^{*} \subseteq A^{\StructureListOneStep{m}}.
  \]
  Hence, \(\StructureListOneStep{m} \models A(X) \rightarrow A(\ConsMult{n_{1},\dots,n_{m}}{l})\).
  For \(l \in \mathcal{N}\) we show the contrapositive of the induction step.
  Suppose first that \((n_{1},\dots,n_{m}) \frown l \notin A^{\StructureListOneStep{m}}\).
  Hence, we have \((n_{1}, \dots, n_{k}) \frown l = N_{k}\) for some \(k\in \NaturalNumbers\) and \(\divides{m}{k}\).
  Thus, \(l = N_{k+m}\) that is \(l \notin A^{\StructureListOneStep{m}}\).
  Hence, \(\StructureListOneStep{m} \models A(x) \rightarrow A(\ConsMult{x_{1},\dots,x_{m}}{X})\).
  However, we also have \(N_{0} \not \in A^{\StructureListOneStep{m}}\).
  Hence, \(\StructureListOneStep{m} \not \models \ListBigStepInductionAxiom{x}{m}{A(x)}\).
\end{proof}
\begin{proof}[Proof of \cref{pro:two_step:main_claim}]
  An immediate consequence of \cref{lem:two_step:structure_satisfies_base_theory,lem:two_step:structure_satisfies_open_induction}, and \cref{lem:two_step_induction_counterexample}.
\end{proof}
So far we have shown that simulating quantifier-free \(m + 1\)-step induction over lists with \(m\)-step induction, is not possible when induction formulas are quantifier-free.
The simulation of big-step induction in \cref{lem:12} by one-step induction makes use of universal quantifiers and conjunction.
This gives rise to the question whether the use of conjunction is necessary.
We conjecture that it is necessary in the following sense.
By \(\ClausesOf{\Language}\) we denote the set of all clauses (disjunctions of atoms and their negation) over the language \(\Language\).
\begin{conjecture}
  Let \(m \geq 2\), then \[
    \TheoryListBase + \bigcup_{1 \leq j < m}\ListBigStepInductionSchema{\forall_{1}\ClausesOf{\LanguageListA}}{j} \not \vdash \ListBigStepInductionAxiom{x}{m}{A(x)}.
  \]
\end{conjecture}
This conjecture is particularly interesting for the methods presented in \cite{reger2019,hajdu2020,hajdu2021b}.
As shown in \cite{hetzl2023apal,vierling2022} these methods carry out induction on literals and clauses.
However, the results in \cite{hetzl2023apal,vierling2022} are formulated for induction over natural numbers and need to be adapted to the case for induction over lists and other recursive datatypes.
A positive answer to the conjecture above together with analogues of the results \cite{hetzl2023apal,vierling2022} would provide a formal justification for the necessity to implement more powerful induction rules that handle conjunction and quantification such as described in \cite{hajdu2021b}.
As a byproduct, the formulas \(\ListBigStepInductionAxiom{X}{m}{A(X)}\) with \(m \geq 1\) form a set of benchmark problems of increasing difficulty for automated theorem provers.

The above shows that mechanizing induction on lists is more complicated than induction on natural numbers in the sense that a reduction of big-step induction to one-step induction requires induction formulas with a higher quantifier-complexity.
In the following we will consider lists with a concatenation operation and we will show that big-step induction does not prove the right cancellation of concatenation.
\section{Right cancellation of list concatenation}
\label{sec:right_cancellation}
In the previous section we have shown that quantifier-free \((m+1)\)-big step induction is strictly stronger than quantifier-free \(m\)-step induction, but not stronger than \(\forall_{1}\) induction.
In this section we show that big-step quantifier-free induction is in general strictly weaker than \(\forall_{1}\) induction.
We will prove this result by showing that the right cancellation property of the append operation on lists can not be proved with quantifier-free big-step induction on lists.
This result is of particular interest for the automation of proof by mathematical induction, since it implies the necessity to work with induction rules that exceed the power quantifier-free big-step induction to handle comparatively basic properties such as the right cancellation of list concatenation.

In the following we will work with a language that extends the base language of lists \(\LanguageListBase\) by an infix symbol for the concatenation of lists.
We will work with the usual left-recursive definition of concatenation.
\begin{definition}
  \label{def:theory_list_append}
  The infix function symbol \(\ListAppend{\cdot}{\cdot} : \SortList \times \SortList \to \SortList\) represents the append operation on lists.
  We define the language \(\LanguageListAppend\) to be \(\LanguageListBase \cup \{ \frown \}\).
  The theory \(\TheoryListAppend\) extends the base theory of lists \(\TheoryListBase\) by the following axioms
  \begin{gather}
    \ListAppend{\Nil}{Y} = Y, \label{ax:thlistappend:1} \tag{L1.1}
    \\
    \ListAppend{\Cons{x}{X}}{Y} = \Cons{x}{\ListAppend{X}{Y}}. \label{ax:thlistappend:2} \tag{L1.2}
  \end{gather}
\end{definition}
In the following lemmas we prove several properties about lists, and in particular the concatenation operation, using increasingly powerful induction principles.
We start by proving some simple properties with quantifier-free induction.
\begin{lemma}
  \label{lem:8}
  The theory \(\TheoryListAppend + \ListInductionSchema{\Open(\LanguageListAppend)}\) proves the following formulas
  \begin{enumerate}[label=(\roman*)]
  \item \(X \frown \Nil = X\),
  \item \(X \frown (Y \frown Z) = (X \frown Y) \frown Z\).
  \end{enumerate}  
\end{lemma}
\begin{proof}
  For both formulas we use a straightforward induction on \(X\) and making use of \eqref{ax:thlistappend:1}, \eqref{ax:thlistappend:2}.
\end{proof}
We prove the next property, the right cancellation for single-element lists, using simultaneous induction on two variables.
\begin{definition}
  \label{def:double_induction}
  Let \(\varphi(X,Y,\vec{z})\) be a formula, then the formula \(\ListDoubleInductionAxiom{X,Y}{\varphi}\) is given by
  \begin{multline*}
    \left(
    \begin{split}
      & \Forall{X}{\varphi(X,\Nil,\vec{z})} \wedge \Forall{Y}{\varphi(\Nil,Y,\vec{z})} \\
      & \wedge \Forall{X}{\Forall{Y}{\Forall{x}{\Forall{y}{\left(\varphi(X,Y,\vec{z}) \rightarrow \varphi(\Cons{x}{X},\Cons{y}{Y},\vec{z})\right)}}}}
    \end{split}
    \right)
    \\
      \rightarrow \Forall{X}{\Forall{Y}{\varphi(X,Y,\vec{z})}}.
    \end{multline*}
    Let \(\Gamma\) be a set of formulas, then the theory \(\ListDoubleInductionSchema{\Gamma}\) is axiomatized by the sentences \(\Forall{\vec{z}}{\ListDoubleInductionAxiom{X,Y}{\varphi(X,Y,\vec{z})}}\) with \(\varphi(X,Y,\vec{z}) \in \Gamma\).
\end{definition}
\begin{lemma}
  \label{lem:9}
  \(\TheoryListAppend + \ListDoubleInductionSchema{\Open(\LanguageListAppend)}\) proves \[
    Y \frown \Cons{x}{\Nil} = Z \frown \Cons{x}{\Nil} \rightarrow Y = Z.
  \]
\end{lemma}
\begin{proof}
  We proceed by induction on \(Y\) and \(Z\) simultaneously.
  We consider only one of the two base cases, since the other one is symmetric.
  For the base case \(Y = \Nil\) we assume \(\Nil \frown \Cons{x}{\Nil} = Z \frown \Cons{x}{\Nil}\) and we have to show that \(Z = \Nil\).
  First of all, by \eqref{ax:thlistappend:1} we obtain \(\Cons{x}{\Nil} = Z \frown \Cons{x}{\Nil}\).
  By \cref{lem:11} we can consider two cases.
  If \(Z = \Nil\), then we are done.
  Otherwise, there are \(z'\) and \(Z'\) such that \(Z = \Cons{z'}{Z'}\).
  Thus
  \begin{align*}
    \Cons{x}{\Nil}
    & = \Cons{z'}{Z'} \frown \Cons{x}{\Nil}
    \\
    & =_{\eqref{ax:thlistappend:2}} \Cons{z'}{Z' \frown \Cons{x}{\Nil}}.
  \end{align*}
  Therefore, by \eqref{ax:thlistbase:2} we have in particular \(\Nil = Z' \frown \Cons{x}{\Nil}\).
  We apply \cref{lem:11} and consider two cases.
  If \(Z' = \Nil\), then \(\Nil = \Nil \frown \Cons{x}{\Nil} = \Cons{x}{\Nil}\), which contradicts \eqref{ax:thlistbase:1}.
  Otherwise, there are \(z^{\prime\prime}\) and \(Z^{\prime\prime}\) such that \(Z' = \Cons{z^{\prime\prime}}{Z^{\prime\prime}}\), then, by \eqref{ax:thlistappend:2},
  \(\Nil = \Cons{z^{\prime\prime}}{Z^{\prime\prime}\frown \Cons{x}{\Nil}}\), which contradicts \eqref{ax:thlistbase:1}.
  For the induction step assume \(Y \frown \Cons{x}{\Nil} = Z \frown \Cons{x}{\Nil} \rightarrow Y = Z\) and \(\Cons{y}{Y} \frown \Cons{x}{\Nil} = \Cons{z}{Z} \frown \Cons{x}{\Nil}\).
  Then by \eqref{ax:thlistappend:2} and \eqref{ax:thlistbase:2} we obtain \(y = z\) and
  \[
    Y \frown \Cons{x}{\Nil} = Z \frown \Cons{x}{\Nil}.
  \]
  By the induction hypothesis we obtain \(Y = Z\), thus, \(\Cons{y}{Y} = \Cons{z}{Z}\).
\end{proof}
Observe that double induction is contained within induction on \(\forall_{1}\) formulas when working modulo case analysis \(\mathrm{CA}\) given by
\[
  \Forall{X}{\left(X = \Nil \vee \Exists{X'}{\Exists{x'}{X = \Cons{x'}{X'}}}\right)}.
\]
\begin{lemma}
  \label{lem:10}
  \(\mathrm{CA} + \ListInductionSchema{\forall_{1}(\Language)} \vdash \ListDoubleInductionSchema{\Open(\Language)}\).
\end{lemma}
\begin{proof}
  Let \(\varphi(X,Y,\vec{z})\) be a quantifier-free \(\Language\) formula.
  Let \(X, Y, \vec{z}\) be fixed and assume \(\Forall{X}{\varphi(X,\Nil,\vec{z})}\), \(\Forall{Y}{\varphi(\Nil,Y,\vec{z})}\), and \[
    \Forall{X}{\Forall{Y}{\Forall{x}{\Forall{y}{\left(\varphi(X,Y,\vec{z}) \rightarrow \varphi(\Cons{x}{X},\Cons{y}{Y},\vec{z})\right)}}}}.
  \]
  We proceed by induction on \(X\) in \(\Forall{Y}{\varphi(X,Y,\vec{z})}\).
  The base case follows immediately from the assumptions.
  For the step case assume \(\Forall{Y}{\varphi(X,Y,\vec{z})}\) and let \(Y\) be fixed.
  By \(\mathrm{CA}\) we can consider two cases.
  If \(Y = \Nil\), then we are done by the assumption.
  Otherwise, there are \(y'\) and \(Y'\) such that \(Y = \Cons{y'}{Y'}\).
  By the induction hypothesis, we obtain \(\varphi(X, Y', \vec{z})\).
  Hence, by the third assumptions, we have \(\varphi(\Cons{x}{X},  \Cons{y'}{Y'}, \vec{z})\), that is, \(\varphi(\Cons{x}{X}, Y, \vec{z})\).
\end{proof}
Using induction on a \(\forall_{1}\) formula, we can straightforwardly prove the right cancellation of the append operation for arbitrary lists.
\begin{lemma}
  \label{lem:6}
  The theory \(\TheoryListAppend + \ListInductionSchema{\forall_{1}(\LanguageListAppend)}\) proves \[
    Y \frown X = Z \frown X \rightarrow Y = Z.
  \]
\end{lemma}
\begin{proof}  
  We proceed by induction on \(X\) in the formula \[
    \Forall{Y}{\Forall{Z}{\left(Y \frown X = Z \frown X \rightarrow Y = Z\right)}}.
  \]
  For the base case, let \(Y\) and \(Z\) be arbitrary and assume \(Y \frown \Nil = Z \frown \Nil\).
  By \cref{lem:8} we readily obtain \(Y = Z\).
  For the step case we assume \[
    \Forall{Y}{\Forall{Z}{Y \frown X = Z \frown X \rightarrow Y = Z}}.
  \]
  and \(Y \frown \Cons{x}{X} = Z \frown \Cons{x}{X}\).
  By \eqref{ax:thlistappend:1}, \eqref{ax:thlistappend:2}, and \cref{lem:8} we obtain
  \begin{gather*}
    (Y \frown \Cons{x}{\Nil}) \frown X = (Z \frown \Cons{x}{\Nil}) \frown X.
  \end{gather*}
  By the induction hypothesis we obtain \(Y \frown \Cons{x}{\Nil} = Z \frown \Cons{x}{\Nil}\).
  Hence, by \cref{lem:9,lem:10} we obtain \(Y = Z\).
\end{proof}
In the remainder of this section we will show that right cancellation of append cannot be proved by quantifier-free big-step induction on lists.
\begin{theorem}
  \label{thm:open_ind_not_proves_right_cancellation}
  \[
    \TheoryListAppend + \bigcup_{m \in \NaturalNumbers}\ListBigStepInductionSchema{\Open(\LanguageListAppend)}{m+1} \not \vdash Y \frown X = X \rightarrow Y = \Nil.
  \]
\end{theorem}
We proceed as usual by constructing a structure that satisfies the base theory of lists with append together with quantifier-free induction for lists, but which contains elements \(l_{1}, l_{2}\) such that \(l_{1} \frown l_{2} = l_{2}\) and \(l_{1} \neq \varepsilon\).
Since the concatenation of transfinite sequences of length greater or equal to \(\omega\) does not have the right cancellation property, as for example \(a \frown a^{\omega} = a^{\omega}\), it seems natural to use concatenation as an interpretation of the append symbol \(\ListAppendSymbolUntyped\).

In \cref{sec:two_step} we have already mentioned that, in order to construct a model
 of \(\TheoryListBase + \ListInductionSchema{\Open(\LanguageListBase)}\) we have to avoid transfinite sequences \(\lambda\) such that \(\lambda = w \frown \lambda\) for some \(w \in \NaturalNumbers^{*}\), cf.~\cref{lem:11}.
However, we may introduce sequences that have a transfinitely periodic structure, such as, the sequence \(N_{0}^{\omega} = N_{0} \frown N_{0}^{\omega}\) of length \(\omega^{2}\).

In the following we define the set of elements that we will use for the construction of the model of quantifier-free big-step induction.
\newcommand{\StructureListOpen}{M_{2}}
\newcommand{\ListAppendInt}[3]{#2 \frown^{#1} #3}
\newcommand{\ListConsInt}[3]{\ConsSymbol^{#1}(#2,#3)}
\newcommand{\ListNilInt}[1]{\Nil^{#1}}
\begin{definition}
  \label{def:right_canc:structure}
  The structure \(\StructureListOpen\) interprets the sort \(\SortIndividuals\) as the set \(\NaturalNumbers\) and the sort \(\SortList\) as the set \(\mathfrak{L}\) given by
  \[
    \left\{ \flatten{\mathfrak{l}} \frown w \mid w \in \NaturalNumbers^{*}, \mathfrak{l} \in \mathcal{N}^{\beta}, \beta < \omega^{2}\right\}.
  \]
  Furthermore, the structure \(\StructureListOpen\) interprets the non-logical symbols as follows
  \begin{gather*}
    \Nil^{\StructureListOpen} \coloneqq \varepsilon,
    \\
    \ListConsInt{\StructureListOpen}{n}{l} \coloneqq n \frown l,
    \\
    \ListAppendInt{\StructureListOpen}{l_{1}}{l_{2}} \coloneqq l_{1} \frown l_{2}.
  \end{gather*}
\end{definition}
We will now first ensure that the structure \(\StructureListOpen\) defined above is indeed a well-defined \(\LanguageListAppend\) structure, that is, that it is closed under the functions \(\NilInterpreted{\StructureListOpen}\), \(\ConsSymbol^{\StructureListOpen}\), and \(\ListAppendSymbolUntyped^{\StructureListOpen}\).
\begin{lemma}
  \label{lem:2}
  \(\StructureListOpen\) is an \(\LanguageListAppend\) structure.
\end{lemma}
\begin{proof}
  We have to show that \(\StructureListOpen\) is closed under the operations \(\Nil^{\StructureListOpen}\), \(\ListConsInt{\StructureListOpen}{\cdot}{\cdot}\), and \(\ListAppendInt{\StructureListOpen}{\cdot}{\cdot}\).
  We have \(\ListNilInt{\StructureListOpen} = \varepsilon \in \NaturalNumbers^{*} \subseteq \mathfrak{L}\).
  Now let \(n \in \NaturalNumbers\) and \(l \in \mathfrak{L}\).
  Let \(l = \flatten{(m_{\gamma})_{\gamma \leq \beta}} \frown w\) with \(\beta < \omega^{2}\), \(m \in \mathcal{N}^{\beta}\), and \(w \in \NaturalNumbers^*\).
  If \(\beta = 0\), then \(n \frown l = n \frown \varepsilon \frown w = n \frown w \in \NaturalNumbers^{*} \subseteq \mathfrak{L}\).
  Otherwise, if \(0 < \beta\), then for \(\gamma < \beta\) we let \[
    m_{\gamma}' \coloneqq
    \begin{cases}
      (n) \frown m_{0} & \text{if \(\gamma = 0\),}
      \\
      m_{\gamma} & \text{otherwise}
    \end{cases}
  \]
  Now observe that \((n) \frown \flatten{(m_{\gamma})_{\gamma < \beta}} = \flatten{( m'_{\gamma})_{\gamma < \beta}}\) and clearly \(m'_{\gamma} \in \mathcal{N}\), for all \(\gamma < \beta\).
  Hence, \(\ListConsInt{\StructureListOpen}{n}{l} \in \mathfrak{L}\).
  Now let \(l_{1}, l_{2} \in \mathfrak{L}\) and consider \(\ListAppendInt{\StructureListOpen}{l_{1}}{l_{2}}\).
  If \(l_{1} \in \NaturalNumbers^{*}\), then we use an analogous argument as above.
  If \(l_{2} \in \NaturalNumbers^{*}\), then we clearly have \(l_{1} \frown l_{2} \in \mathfrak{L}\).
  If \(l_{1}\) and \(l_{2}\) are non-standard, then for \(i = 1, 2\) there are \(\alpha_{i} < \omega^{2}\), \(\mathfrak{a}_{i} \in \mathcal{N}^{\alpha_{i}}\), \(w_{i} \in \NaturalNumbers^{*}\) such that \(l_{i} = \flatten{\mathfrak{a}_{i}} \frown w_{i}\).
  Moreover, there exists \(\delta \leq \alpha_{2}\) and \(w' \frown N_{k} \in \mathcal{N}\) such that \(1 + \delta = \alpha_{2}\) and \(l_{2} = w' \frown N_{k} \frown \flatten{(\mathfrak{a}_{2,1 + \gamma})_{\gamma < \delta}}\).
  Therefore, we have \[
    l_{1} \frown l_{2} = \flatten{\mathfrak{a}_{1}} \frown (w_{1} \frown w' \frown N_{k}) \frown \flatten{(\mathfrak{a}_{2,1 + \gamma})_{\gamma < \delta}} \frown w_{2}.
  \]
  Since \(w_{1} \frown w' \frown N_{k} \in \mathcal{N}\) and \(\beta_{1} + \beta_{2} < \omega^{2}\) we have \(l_{1} \frown l_{2} \in \mathfrak{L}\).
\end{proof}
Next we show that \(\StructureListOpen\) satisfies the basic axioms of the list constructors \(\Nil\) and \(\ConsSymbol\), as well as those of the append symbol. 
\begin{lemma}
  \label{lem:3}
  \(\StructureListOpen \models \TheoryListAppend\).
\end{lemma}
\begin{proof}
  Let \(n \in \NaturalNumbers\) and \(l \in \mathfrak{L}\), then there is some ordinal \(\alpha < \omega^{3}\) such that \(l \in \NaturalNumbers^{\alpha}\).
  Hence, \(\ListConsInt{\StructureListOpen}{n}{l} \in \NaturalNumbers^{1 + \alpha}\).
  Therefore \(\ListConsInt{\StructureListOpen}{n}{l} \neq \ListNilInt{\StructureListOpen} = \varepsilon = \varnothing\).
  Thus \(\StructureListOpen \models \eqref{ax:thlistbase:1}\).
  Now let \(n_{1}, n_{2} \in \NaturalNumbers\) and \(l_{1}, l_{2} \in \mathfrak{L}\) and assume that \(n_{1} \frown l_{1} = n_{2} \frown l_{2}\).
  For \(i = 1, 2\), let \(\alpha_{i} < \omega^{3}\) such that \(l_{i} \in \NaturalNumbers^{\alpha_{i}}\).
  We thus have \(1 + \alpha_{1} = 1 + \alpha_{2}\) which implies \(\alpha_{1} = \alpha_{2}\).
  Therefore, \(n_{1} = (n_{1} \frown l_{1})_{0} = (n_{2} \frown l_{2})_{0} = n_{2}\).
  Let \(\gamma < \alpha_{1}\), then \(l_{1,\gamma} = (n_{1} \frown l_{1})_{1 + \gamma} = (n_{2} \frown l_{2})_{1 + \gamma} = l_{2,\gamma}\).
  Thus, \(l_{1} = l_{2}\).
  Hence \(\StructureListOpen \models \eqref{ax:thlistbase:2}\).
  Now let \(l \in \mathfrak{L}\).
  We have \(\ListNilInt{\StructureListOpen} \frown l =  \varepsilon \frown l = l\).
  Hence, \(\StructureListOpen \models \eqref{ax:thlistappend:1}\).
  Now let \(n \in \NaturalNumbers\), \(l, l' \in \mathfrak{L}\).
  Then we have
  \[
    \ListAppendInt{\StructureListOpen}{\ListConsInt{\StructureListOpen}{n}{l}}{l'} = ((n) \frown l) \frown l' = (n) \frown (l \frown l') = \ListConsInt{\StructureListOpen}{n}{\ListAppendInt{\StructureListOpen}{l}{l'}}.
  \]
  Thus, \(\StructureListOpen \models \eqref{ax:thlistappend:2}\).
\end{proof}
Since the domain of \(\StructureListOpen\) interprets the sort of lists as transfinite sequences and the append operation as the concatenation of transfinite sequences, we can decompose \(\SortList\) terms as follows.
\begin{lemma}
  \label{lem:4}
  Let \(t(X,\vec{y})\) be a \(\LanguageListAppend\) \(\SortList\)-term and \(\vec{b}\) elements of \(\StructureListOpen\), then there exist \(n \in \NaturalNumbers\) and \(l_{0}, \dots, l_{n} \in \mathfrak{L}\) such that
  \[
    \StructureListOpen \models t(X,\vec{b}) = l_{0} \frown X \frown l_{1} \frown \dots \frown l_{n-1} \frown X \frown l_{n}.
  \]
\end{lemma}
\begin{proof}
  We proceed by induction on the structure of the term \(t\).
  If \(t\) is \(\Nil\), then \(\StructureListOpen \models t = \varepsilon\), and thus we are done.
  If \(t\) is the variable \(X\), then we are done by letting \(n = 0\) and \(l_{0} = \varepsilon \in \NaturalNumbers^{0}\).
  If \(t\) is of the form \(\Cons{u}{t'}\), then \(\StructureListOpen \models u(\vec{b}) = k\), for some \(k \in \NaturalNumbers\).
  Hence we apply the induction hypothesis in order to obtain \(n' \in \NaturalNumbers\) and \(l_{0}', \dots, l_{n'}' \in \mathfrak{L}\) such that \(\StructureListOpen \models t'(X,\vec{b}) = l_{0}' \frown X \frown \dots \frown l_{n' -1}' \frown X \frown l_{n'}'\).
  Hence,
  \[
    \StructureListOpen \models t(X,\vec{b}) =  (k) \frown l_{0}' \frown X \frown \dots \frown l_{n'-1}' \frown X \frown l'_{n'}.
  \]
  Thus, we let \(n = n'\) and \(l_{0} = (k)) \frown l_{0}'\) and \(l_{i} = l_{i}'\) for \(1 \leq i \leq n\).
  If \(t\) is of the form \(t_{1} \frown t_{2}\), then simply apply the induction hypothesis to \(t_{1}\) and \(t_{2}\).
\end{proof}
Equational predicates over \(\StructureListOpen\) in one variable stabilize eventually in a similar way to \cref{lem:list_equations_stabilize}.
\begin{lemma}
  \label{lem:5}
  Let \(E(X)\) be an \(\LanguageListAppend(\StructureListOpen)\) equation such that \(\StructureListOpen \not \models E(X)\), then there exists \(N \in \NaturalNumbers\) such that \(\StructureListOpen \not \models E((n) \frown l)\) for all \(n \geq N\) and \(l \in \mathfrak{L}\).
\end{lemma}
\begin{proof}
  Let \(E(X)\) be \(t_{1}(X) = t_{2}(X)\), then by Lemma~\ref{lem:4} for \(i =1, 2 \) there exist \(n_{i} \in \NaturalNumbers\) and \(l_{0}^{i},\dots,l_{n_{i}}^{i} \in \mathfrak{L}\) such that
  \[
    \StructureListOpen \models t_{i} = l_{0}^{i} \frown X \frown \dots \frown l_{n_{i}-1}^{i} \frown X \frown l_{n_{i}}^{i}.
  \]
  By the symmetry of equality we can assume \(n_{1} \leq n_{2}\) without loss of generality.
  Since \(\StructureListOpen \not \models E(X)\) we either have \(n_{1} \neq n_{2}\) or \(l_{i}^{1} \neq l_{i}^{2}\) for some \(i \in \{ 0, \dots, n_{1}\}\).
  We start by assuming that \(l_{i}^{1} = l_{i}^{2}\) for \(i = 0, \dots, n_{1}\) and \(n_{1} < n_{2}\).
  Then by the left cancellation of \(\frown\) we obtain
  \begin{equation*}
    \StructureListOpen \models E(X) \leftrightarrow \varepsilon = X \frown l_{n_{1}+1}^{2} \frown \dots \frown l_{n_{2}-1}^{2} \frown X \frown l_{n_{2}}^{2}.
  \end{equation*}
  Hence, we have \(\StructureListOpen \not \models E((n) \frown l)\), for all \(n \in \NaturalNumbers\) and \(l \in \mathfrak{L}\).
  Now consider the case where there exists \(j \in \{0, \dots, n_{1}\}\) such that \(l_{j}^{1} \neq l_{j}^{2}\) and let \(j_{0} \in \{ 0, \dots, n_{1}\}\) be the least such number.
  There are sequences \(l, l_{j_{0}}^{1\prime}\) and \(l_{j_{0}}^{2\prime}\) such that \(l_{j_{0}}^{i} = l \frown l_{j_{0}}^{i\prime}\) for \(i = 1, 2\) and either \(|l_{j_{0}}^{1\prime}| = 0\), \(|l_{j_{0}}^{2\prime}| \geq 1\), or \(|l_{j_{0}}^{1\prime}| \geq 1\), \(|l_{j_{0}}^{2\prime}| = 0\), or \(|l_{j_{0}}^{1'}| \geq 1\), \(|l_{j_{0}}^{2'}| \geq 1\) and \((l_{j_{0}}^{1'})_{0} \neq (l_{j_{0}}^{2'})_{0}\).
  Hence, by left cancellation of concatenation, we obtain
  \begin{multline*}
    \StructureListOpen \models E(X) \leftrightarrow l_{j_{0}}^{1\prime} \frown X \frown \dots \frown l_{n_{1}-1}^{1} \frown X \frown l_{n_{1}} = \\ l_{j_{0}}^{2\prime} \frown X \frown \dots \frown l_{n_{2}-1}^{2} \frown X \frown l_{n_{2}}.
  \end{multline*}
  If \(l_{j_{0}}^{1'} = \varepsilon\) and \(l_{j_{0}}^{2'} \neq \varepsilon\), then for \(n \neq (l_{j_{0}}^{2})_{0}\), we have \(\StructureListOpen \not \models E((n) \frown l)\) for all \(l \in \mathfrak{L}\).
  The case where \(l_{j_{0}}^{1'} \neq \varepsilon\) and \(l_{j_{0}}^{2'} = \varepsilon\) is symmetric.
  Finally, in the case that \(l_{j_{0}}^{1'}, l_{j_{0}}^{2'} \neq \varepsilon\) with \((l_{j_{0}}^{1'})_{0} \neq (l_{j_{0}}^{2'})\), we trivially have \(\StructureListOpen \not \models E(l)\) for all \(l \in \mathfrak{L}\).
\end{proof}
As an immediate consequence of the previous lemma, we obtain the following result, which essentially says that for a non-standard element \(\lambda\) a \(\SortList\)-equation \(E(X)\) can eventually be stabilized for predecessors of \(\lambda\).
\begin{lemma}
  \label{lem:7}
  Let \(E(X)\) be an \(\LanguageListAppend(\StructureListOpen)\) equation such that \(\StructureListOpen \not \models E(X)\) and \(\lambda \in \mathfrak{L} \setminus \NaturalNumbers^{*}\).
  Then there exists \(N \in \NaturalNumbers\) such that \(\StructureListOpen \not \models E(\lambda \uparrow n)\) for all \(n \geq N\).
\end{lemma}
\begin{proof}
  First by applying \cref{lem:5} we obtain \(m_{0}\) such that \(M \not \models E((m) \frown l)\) for all \(m \geq m_{0}\) and \(l \in \mathfrak{L}\).
  Since \(\lambda \not\in \NaturalNumbers^{*}\), there clearly is \(n_{0} \in \NaturalNumbers\) such that \(\lambda \uparrow n_{0} = N_{m_{0}} \frown \lambda'\) for some \(\lambda' \in \mathfrak{L}\).
  Since \((\lambda \uparrow N + k)_{0} = m_{0} + k \geq m_{0}\) for \(k \in \NaturalNumbers\), we have \(\StructureListOpen \not \models E(\lambda \uparrow n)\) for all \(n \geq n_{0}\).
\end{proof}
The previous two lemmas show that the truth value of formulas in \(\StructureListOpen\) on non-standard elements eventually synchronizes with that on standard elements, when considering sufficiently distant predecessors.
\begin{lemma}
  \label{lem:1}
  Let \(\varphi(X)\) be an open \(\LanguageListAppend(\StructureListOpen)\) formula and \(\lambda \in \mathfrak{L}\), then there exists \(n_{0} \in \NaturalNumbers\) such that
  \[
    \StructureListOpen \models \varphi(\lambda \uparrow n) \leftrightarrow \varphi((n)),
  \]
  for all \(n \geq n_{0}\).
\end{lemma}
\begin{proof}
  Clearly, it suffices to consider the \(\SortList\)-equations of \(\varphi\), since the \(\SortIndividuals\)-equations do not depend on the variable \(X\).
  Let \(E_{1}(X), \dots, E_{k}(X)\) be the atoms of \(\varphi\) with \(\StructureListOpen \not \models E_{i}(X)\), for \(i = 1, \dots, k\).
  Then by \cref{lem:7,lem:5} there is \(n_{0} \in \NaturalNumbers\) such that \(M_{2} \not \models E_{i}(\lambda \uparrow n)\) and \(M_{2} \not \models E_{i}((n))\) for \(n \geq n_{0}\) and \(i = 1, \dots, k\).
  Since we have \(M_{2} \models E(X)\) for the other \(\SortList\)-atoms of \(\varphi\), we obtain \(M_{2} \models \varphi(\lambda \uparrow n) \leftrightarrow \varphi((n))\) for \(n \geq n_{0}\).
\end{proof}
We are now ready to show that \(\StructureListOpen\) satisfies open big-step induction.
\begin{proposition}
  \label{pro:2}
  Let \(m \in \NaturalNumbers\) with \(m \geq 1\), then \(\StructureListOpen \models \ListBigStepInductionSchema{\Open(\LanguageListAppend)}{m}\).
\end{proposition}
\begin{proof}
  Let \(\Formula(X)\) be a quantifier-free \(\LanguageListAppend(\StructureListOpen)\) formula.
  Assume that
  \begin{gather*}
    \StructureListOpen \models \bigwedge_{i = 1, \dots, m}\Formula(\ConsMult{x_{1}, \dots, x_{i-1}}{\Nil}), \label{assumption:IB} \tag{\(\ast\)}
    \\
    \StructureListOpen \models \Formula(X) \rightarrow \Formula(\ConsMult{x_{1},\dots,x_{m}}{X}). \label{assumption:IS} \tag{\(\star\)}
  \end{gather*}
  Let \(\lambda \in \mathfrak{L}\).
  If \(\lambda \in \NaturalNumbers^{*}\), then a straightforward induction making use of \eqref{assumption:IB} and \eqref{assumption:IS} yields \(M_{2} \models \varphi(\lambda)\).
  Now we consider the case \(\lambda \notin \NaturalNumbers^{*}\), that is, \(\lambda\) is a non-standard element.
  By \cref{lem:1} there is \(n_{0} \in \NaturalNumbers\) such that \(\StructureListOpen \models \Formula(\lambda \uparrow n)\) if and only if \(\StructureListOpen \models \Formula((n))\) for all \(n \geq n_{0}\).
  In particular, we thus have \[
    \StructureListOpen \models \Formula(\lambda \uparrow n_{0} + m + i) \leftrightarrow \Formula((n_{0} + m + i))
  \] for \(i = 0, \dots, m -1\).
  Since, \(\StructureListOpen \models \Formula(w)\) for all \(w \in \NaturalNumbers^{*}\), we obtain \(\StructureListOpen \models \Formula(\lambda \uparrow n_{0} + m + i)\) for \(i = 0, \dots, m -1\).
  By a straightforward induction starting with \(\StructureListOpen \models \Formula(\lambda \uparrow n_{0} + m - 1)\), \dots, \(\StructureListOpen \models \Formula(\lambda \uparrow n_{0})\) and making use of \eqref{assumption:IS} we obtain \(\StructureListOpen \models \Formula(w \frown (\lambda \uparrow n_{0}))\) for all \(w \in \NaturalNumbers^{*}\).
  Therefore, we have in particular \(\StructureListOpen \models \Formula(\lambda)\).
\end{proof}
\begin{proof}[Proof of \cref{thm:open_ind_not_proves_right_cancellation}]
  Clearly, \(N_{0} \in \mathfrak{L}\).
  Since \(N_{0}^{\omega} = \flatten{(N_{0})_{\gamma < \omega}}\), we have \(N_{0}^{\omega} \in \mathfrak{L}\).
  Now observe that \(N_{0} \frown N_{0}^{\omega} = N_{0}^{\omega} \) but \(N_{0} \neq \varnothing\).
  Hence, by \cref{pro:2} we are done.
\end{proof}
This result is of interest for automated inductive theorem proving, because it essentially provides a lower bound on the power necessary for the proof of a rather simple yet practically relevant property about the important datatype of lists.

The unprovability of right cancellation of concatenation is a first step towards a classification of the inductive power needed to prove certain practically interesting properties of finite Lisp-like lists.
\Cref{thm:open_ind_not_proves_right_cancellation} as well as the auxiliary results of this section give rise to many related questions and conjectures that we will briefly discuss in the following.

We conjecture that even quantifier-free simultaneous induction on several variables with big-steps does not prove right cancellation of the concatenation operation.
Let \(\vec{x} = (x_{1}, \dots, x_{n})\) be a finite sequence and \(i \in \NaturalNumbers\) such that \(1 \leq i \leq n\), then by \(\vec{x}_{< i}\) we denote the sequence \((x_{1}, \dots, x_{i -1})\).
Similarly, \(\vec{x}_{>i}\) denotes the sequence \((x_{i+1}, \dots, x_{n})\).
\newcommand{\ListDiagonalBigStepHeterogeneousInductionAxiom}[3]{
  I_{#1 \curvearrowright #2}^{\mathsf{list}}#3
}
\newcommand{\ListDiagonalBigStepHeterogeneousInductionSchema}[1]{
  {#1}\text{-}\mathrm{IND}^{\mathsf{list}}_{\nearrow_{\curvearrowright}}
}
\begin{definition}
  \label{def:list_diagonal_big_step_heterogeneous_induction}
  Let \(\vec{X} = (X_{1}, \dots, X_{m})\) be pairwise distinct variables with \(m \geq 1\), \(\vec{p} = (p_{1}, \dots, p_{m})\) a sequence of non-zero natural numbers, and \(\varphi(\vec{X}, \vec{z})\) a formula.
  The multivariate big-step list induction axiom \(\ListDiagonalBigStepHeterogeneousInductionAxiom{\vec{X}}{\vec{p}}{\varphi}\) for \(\varphi\) is given by
  \begin{multline*}
    \left(
      \begin{split}
        \bigwedge_{i = 1}^{m}\bigwedge_{j = 1}^{p_{i}} \Forall{\vec{X}_{<i}}{\Forall{\vec{X}_{>i}}{\Forall{x_{1}, \dots, x_{j-1}}{\varphi(\vec{X}_{<i},\ConsMult{x_{1},\dots,x_{j-1}}{\Nil},\vec{X}_{>i},\vec{z})}}} \\
        \wedge \Forall{\vec{X}}{\Forall{\vec{x}^{p_{1}}}{\dots\Forall{\vec{x}^{p_{m}}}{\left(\varphi(\vec{X},\vec{z}) \rightarrow \varphi(\ConsMult{\vec{x}^{p_{1}}}{X_{1}}, \dots, \ConsMult{\vec{x}^{p_{m}}}{X_{m}}, \vec{z})\right)}}}
      \end{split}
    \right)
    \\
    \rightarrow \Forall{\vec{X}}{\varphi(\vec{X})}.
  \end{multline*}
  where the \(\vec{x}^{p_{i}}\) with \(i \in \{ 1, \dots, m\}\) are vectors of variables of sort \(\SortIndividuals\) whose elements are all pairwise distinct.
  Let \(\SetOfFormulas\) be a set of formulas, then theory \(\ListDiagonalBigStepHeterogeneousInductionSchema{\SetOfFormulas}\) is axiomatized by \(\ListDiagonalBigStepHeterogeneousInductionAxiom{\vec{X}}{\vec{p}}{\varphi}\) with \(\varphi(\vec{X},\vec{z}) \in \SetOfFormulas\) and \(\vec{X}, \vec{p}\) as above.
\end{definition}
\begin{conjecture}
  \label{con:double_induction_right_cancellation}
  \(\TheoryListAppend + \ListDiagonalBigStepHeterogeneousInductionSchema{\Open(\LanguageListAppend)} \not \vdash Y \frown X = Z \frown X \rightarrow Y = Z\).
\end{conjecture}
A positive answer to this question would thus greatly improve upon our \cref{thm:open_ind_not_proves_right_cancellation}.
A related question of interest is whether single-element right cancellation can be proven by quantifier-free big-step induction in one variable.

The subject of \ac{AITP} mainly focuses on the mechanization of induction in general, rather than on the mechanization of individual theories.
Nevertheless, the theories of lists with concatenation considered in this section are of some practical relevance.
Hence, it may be valuable to investigate their mechanization separately.
Because of the homomorphic relation between natural numbers with addition and lists with concatenation, it could be especially interesting to investigate whether simple theories
 of lists such as \(\TheoryListAppend + \ListInductionSchema{\forall_{1}(\LanguageListAppend)}\) have finite axiomatizations analogous to the one shown in \cite{shoenfield1958} for natural numbers with addition.

Finally, let us observe that as an immediate consequence of \cref{pro:2} we obtain the unprovability of right-decomposition of list by open big-step induction.
\begin{corollary}
  \label{cor:1}
  \(\TheoryListAppend + \bigcup_{m \geq 1}\ListBigStepInductionSchema{\Open(\LanguageListAppend)}{m}\) does not prove
  \[
    X = \Nil \vee \Exists{x'}{\Exists{X'}{X = X' \frown \Cons{x'}{\Nil}}}.
  \]
\end{corollary}
\begin{proof}
  Consider the element \(N_{0} \in \DomainOf{\SortList}{\StructureListOpen}\) and observe that \(N_{0} \neq \Nil\) but since \(|N_{0}| = \omega\), we cannot express \(N_{0}\) as \(\lambda \frown (n)\) with \(\lambda \in \NaturalNumbers^{\leq \omega}\) and \(n \in \NaturalNumbers\).
  Now, the claim follows from \cref{pro:2}.
\end{proof}
Clearly, the formula \(X = \Nil \vee \Exists{x'}{\Exists{X'}{X = X' \frown \Cons{x'}{\Nil}}}\) is provable by induction on the formula itself, that is, by \(\exists_{1}\) induction.
This gives rise to the question whether right-decomposition can be proved by \(\forall_{1}\) induction and more generally to the more general question how \(\exists_{1}\) induction and \(\forall_{1}\) induction over lists with concatenation are related.
This question is relevant for \ac{AITP}, since there are systems such as \cite{kersani2013} that are based on \(\exists_{1}\) induction \cite{hetzl2022} and systems such as \cite{cruanes2017} that are based on \(\forall_{1}\) induction \cite[Chapter 5]{vierling2022}.
We plan to investigate this question separately in the future.
\section{Conclusion}
In this article we have shown two main results about induction for lists.
Firstly, in \cref{sec:two_step} we have shown that quantifier-free \((m+1)\)-step induction can in general not be simulated with quantifier-free \(m\)-step induction.
In particular, this result thus renders impossible a reductive implementation of quantifier-free big-step induction in \ac{AITP} systems with an induction mechanism based on quantifier-free induction.
This observation may be relevant for future extensions of systems based on quantifier-free one-step induction mechanism, such as the \ac{AITP} system described in \cite[Section 3.2]{reger2019}. 
The idea is that whenever an induction principle can be reduced to a simpler one, then for the sake of soundness one should consider the reduction.

The second main result of this article, shown in \cref{sec:right_cancellation}, is the unprovability of right cancellation of the concatenation for lists by quantifier-free big-step induction.
Thus automated inductive theorem provers have to implement a comparatively strong induction mechanism in order to the prove seemingly simple property of right cancellation of concatenation.

In the light of the results of \cref{sec:two_step}, a natural choice would be to implement an induction principle that can handle at least \(\forall_{1}\) induction formulas with conjunction.
Such an induction principle permits a reductive implementation of \(\forall_{1}\) big-step induction.
An example of a system implementing such an induction mechanism is the one described in \cite{cruanes2017} and analyzed in \cite[Chapter 5]{vierling2022}.

One direction for future research is to carry out similar investigations focusing on other datatypes, induction principles, and properties.
In principle questions such as the one addressed in \cref{sec:right_cancellation} could be considered for every problem in benchmark suites such as \cite{claessen2015} in order to obtain a classification of the difficulty of the problems that complements empirical results.

Furthermore, the results in this article raise a number of questions and conjectures that we would like to address in the future.
In particular, we would like to investigate \cref{con:double_induction_right_cancellation}, since a positive answer, showing that quantifier-free induction combining, both, simultaneous induction and big-step induction does not prove right cancellation of concatenation, would significantly strengthen the result of \cref{sec:right_cancellation}.
Another interesting question is whether the right injectivity of concatenation (see \cref{lem:9}) can be proved with quantifier-free big-step induction.
Finally, the use of transfinite lists used in this article are reminiscent of streams defined by coinduction.
It could be interesting investigate to which extent the techniques employed for the analysis of \ac{AITP} systems can be transferred to systems that automate the coinduction principle such as \cite{leino2014,einarsdottir2018}.
\begin{acronym}
  \acro{AITP}{automated inductive theorem proving}
\end{acronym}

\bibliography{references.bib}
\bibliographystyle{plain}

\end{document}

%% file: header.tex
\newcommand{\VariableSymbols}{\mathfrak{V}}

\newcommand{\notdivides}[2]{#1 \nmid #2}

\newcommand{\ListInductionAxiom}[2]{I_{#1}{#2}}

\newcommand{\ListBigStepInductionAxiom}[3]{I_{#1\curvearrowright #2}{#3}}

\newcommand{\TwoStepInductionSchema}[1]{{#1}\text{-}\mathrm{IND}_{\curvearrowright 2}}
\newcommand{\ListBigStepInductionSchema}[2]{
  {#1}\text{-}\mathrm{IND}_{\curvearrowright{#2}}
}
\newcommand{\ListInductionSchema}[1]{{#1}\text{-}\mathrm{IND}}
\newcommand{\ListDoubleInductionAxiom}[2]{I_{#1}{#2}}
\newcommand{\ListDoubleInductionSchema}[1]{{#1}\text{-}\mathrm{DIND}}

\newcommand{\Open}{\mathrm{Open}}
\newcommand{\ClausesOf}[1]{\mathrm{Clause}(#1)}

\newcommand{\Var}{\mathit{Var}}

\newcommand{\SetOfFormulas}{\Phi}

\newcommand{\Formula}{\varphi}
\newcommand{\FormulaB}{\psi}

\newcommand{\FOStructure}{M}
\newcommand{\Language}{\mathcal{L}}

\newcommand{\NaturalNumbers}{\mathbb{N}}

\newcommand{\QuantifierParenthesis}[3]{({#1}{#2}){#3}}

\newcommand{\Quantifier}[3]{\QuantifierParenthesis{#1}{#2}{#3}}
\newcommand{\Forall}[2]{\Quantifier{\forall}{#1}{#2}}

\newcommand{\Exists}[2]{\Quantifier{\exists}{#1}{#2}}

\newcommand{\DomainOf}[2]{{#2}({#1})}

\newcommand{\Sort}[1]{\mathsf{#1}}

\newcommand{\SortIndividuals}{\Sort{i}}
\newcommand{\SortList}{\mathsf{list}}

\newcommand{\SortBool}{o}

\newcommand{\ConsSymbol}{\mathit{cons}}
\newcommand{\Cons}[2]{\ConsSymbol(#1,#2)}
\newcommand{\ConsMult}[2]{\ConsSymbol(#1;#2)}

\newcommand{\Nil}{\mathit{nil}}
\newcommand{\NilInterpreted}[1]{\Nil^{#1}}
\newcommand{\ListAppend}[2]{#1 \frown #2}
\newcommand{\ListAppendSymbolUntyped}{\frown}

\newcommand{\Theory}{T}
\newcommand{\LanguageListBase}{\Language_{0}}
\newcommand{\LanguageListA}{\Language_{A}}
\newcommand{\TheoryListBase}{T_{0}}
\newcommand{\TheoryListAppend}{\Theory_{1}}
\newcommand{\LanguageListAppend}{\Language_{1}}